\documentclass[]{scrartcl}

\usepackage[]{amsmath, amssymb,amsfonts,amsthm,mathtools,braket,color,enumerate}

\usepackage[margin=25mm]{geometry}

\usepackage{hhline}
\usepackage{url}
\usepackage{here}
\usepackage{stmaryrd}

\usepackage{hyperref}
\hypersetup{
  colorlinks   = true, 
  urlcolor     = blue, 
  linkcolor    = blue, 
  citecolor   = red 
}

\usepackage{tikz-cd}
\usepackage{tikz}
\usetikzlibrary{decorations}
\usetikzlibrary{arrows}
\tikzstyle{v} = [circle, draw, inner sep=2pt, minimum size=3pt, fill=black]
\tikzstyle{l} = [rectangle, draw, rounded corners]

\theoremstyle{plain}
\newtheorem{theorem}{Theorem}[section]
\newtheorem{lemma}[theorem]{Lemma}
\newtheorem{proposition}[theorem]{Proposition}

\theoremstyle{definition}
\newtheorem{definition}[theorem]{Definition}
\newtheorem{conjecture}[theorem]{Conjecture}

\newtheorem{problem}[theorem]{Problem}
\newtheorem{remark}[theorem]{Remark}

\DeclareMathOperator{\Sym}{Sym}
\DeclareMathOperator{\SSym}{SSym}
\DeclareMathOperator{\type}{type}
\DeclareMathOperator{\codim}{codim}

\title {Chromatic Signed-Symmetric Functions of Signed Graphs}
\author{
Masamichi Kuroda
\thanks{Faculty of Engineering, Nippon Bunri University, Oita 870-0316, Japan. 
E-mail:kurodamm@nbu.ac.jp}
\and
Shuhei Tsujie
\thanks{Department of Mathematics, Hokkaido University of Education, Asahikawa, Hokkaido 070-8621, Japan. 
E-mail:tsujie.shuhei@a.hokkyodai.ac.jp}
}

\date{}

\begin{document}
\maketitle
	
\begin{abstract}
\textcolor{red}{One of the main theorems (Theorem \ref{main theorem reciprocity}) was proved by Wolfgang \cite{wolfgang1997two} in 1997.} 
Stanley introduced the chromatic symmetric function of a simple graph, which is a generalization of a chromatic polynomial. 
This is expressed in terms of the integer points of the complements of the corresponding graphic arrangement. 
Stanley proved a combinatorial reciprocity theorem for chromatic functions. 
This is considered as an Ehrhart-type reciprocity theorem for the graphic arrangement. 

We introduce the chromatic signed-symmetric function of a signed graph, an analogue of the chromatic symmetric function, by the integer points of the complements of the corresponding signed-graphic arrangement and prove a generalization of Stanley's reciprocity theorem. 

Stanley has conjectured that the chromatic symmetric function distinguishes trees. 
This conjecture is also generalized for signed trees. 
We verify the conjecture for certain classes of signed paths.
\end{abstract}



\tableofcontents

\section{Introduction}

\subsection{Background}
Let $ \Gamma = (V_{\Gamma}, E_{\Gamma}) $ be a simple graph with vertex set $ V_{\Gamma} $ and edge set $ E_{\Gamma} $. 
We call a map $ \kappa \colon V_{\Gamma} \rightarrow \mathbb{Z}_{>0} $ a \textbf{coloring} of $ \Gamma $. 
A coloring $ \kappa $ is said to be \textbf{proper} if $ \kappa(u) \neq \kappa(v) $ whenever $ \{u,v\} \in E_{\Gamma} $. 

We define a function $ \chi_{\Gamma} $ by for every nonnegative integer $ n $
\begin{align*}
\chi_{\Gamma}(n) \coloneqq \#\Set{\kappa \colon V_{\Gamma} \rightarrow [n] | \kappa \text{ is proper. }}, 
\end{align*}
where $ [n] $ denotes the set $ \{1, \dots, n\} $. 
It is well known that there exists a unique monic polynomial in $ \mathbb{Z}[t] $ such that its evaluation at $ t=n $ coincides with $ \chi_{\Gamma}(n) $. 
Let $ \chi_{\Gamma}(t) $ denote the polynomial and it is called \textbf{chromatic polynomial} of $ \Gamma $. 

For example, we have $ \chi_{K_{3}}(t) = t(t-1)(t-2) $ and $ \chi_{P_{3}}(t) = t(t-1)^{2} $, where $ K_{\ell} $ and $ P_{\ell} $ denote the complete graph and the path on $ \ell $ vertices. 
There are non-isomorphic graphs having the same chromatic polynomial. 
Especially, trees on $ \ell $ vertices have the same chromatic polynomial $ t(t-1)^{\ell-1} $. 

Stanley \cite{stanley1995symmetric-aim} introduced a symmetric function generalization of the chromatic polynomial $ \chi_{\Gamma}(t) $ as follows. 
\begin{align*}
X_{\Gamma} &\coloneqq \sum_{\kappa} \prod_{v \in V_{\Gamma}} x_{\kappa(v)}, 
\end{align*}
where $ x_{1}, x_{2}, \dots  $ are infinitely many indeterminates and the sum runs over all proper colorings $ \kappa \colon V_{\Gamma} \rightarrow \mathbb{Z}_{>0} $. 
Note that if we evaluate $ X_{\Gamma} $ at $ x_{1}=x_{2} = \dots = x_{n} = 1 $ and $ x_{n+1} = x_{n+2} = \dots = 0 $, then the resulting value coincides with $ \chi_{\Gamma}(n) $. 
Therefore $ X_{\Gamma} $ is an invariant of $ \Gamma $ stronger than $ \chi_{\Gamma}(t) $. 

\begin{conjecture}[Stanley {\cite[p.170]{stanley1995symmetric-aim}}]\label{Stanley tree conjecture}
The chromatic symmetric function distinguishes  trees. 
Namely, when $ T_{1} $ and $ T_{2} $ are trees satisfying $ X_{T_{1}} = X_{T_{2}} $, they are isomorphic. 
\end{conjecture}

This conjecture is verified for trees with up to $ 29 $ vertices by Heil and Ji \cite{heil2019algorithm-ajc}. 
For recent studies, see \cite{aliste-prieto2017trees-dm, aliste-prieto2014proper-dm, huryn2020few-iajom, loebl2019isomorphism-adlhpd, martin2008distinguishing-joctsa, orellana2014graphs-dm, smith2015symmetric-a}

There is another conjecture concerning the $ e $-positivity of chromatic symmetric functions. 
See \cite{brosnan2018unit-aim, cho2019positivity-a, cho2019$e$-positivity-sjodm,  cho2020$e$-positivity-sldc, dahlberg2018triangular-a, dahlberg2019new-sldc, dahlberg2020resolving-jotems, dahlberg2018lollipop-sjodm, foley2019classes-tejoc, guay-paquet2013modular-a, guay-paquet2016second-a, harada2018cohomology-sldc, harada2019cohomology-ac, shareshian2016chromatic-aim, tsujie2018chromatic-gac} for recent studies. 
\begin{conjecture}[Stanley and Stembridge {\cite[Conjecture 5.5]{stanley1993immanants-joctsa}}, see also Stanley {\cite[Conjecture 5.1]{stanley1995symmetric-aim}}]
The chromatic symmetric function of the incomparability graph of a $ (\boldsymbol{3}+\boldsymbol{1}) $-free poset is $ e $-positive, that is, the coefficients of the expansion of the chromatic symmetric function with respect to the elementary symmetric functions are nonnegative. 
\end{conjecture}

An \textbf{orientation} of $ \Gamma $ is an assignment of a direction to each edge $ \{u,v\} $, denoted by $ (u,v) $ or $ (v,u) $. 
An orientation is called \textbf{acyclic} if it has no directed cycles. 
A coloring $ \kappa $ is said to be \textbf{compatible} with an orientation $ \mathfrak{o} $ if $ (u,v) \in \mathfrak{o} $, then $ \kappa(u) \geq \kappa
(v) $. 
We define a function $ \overline{\chi}_{\Gamma} $ by 
\begin{align*}
\overline{\chi}_{\Gamma}(n) \coloneqq \#\Set{(\mathfrak{o},\kappa) | \mathfrak{o} \text{ is acyclic and } \kappa \colon V_{\Gamma} \rightarrow [n] \text{ is compatible with } \mathfrak{o}. }. 
\end{align*}

Since $ \chi_{\Gamma} $ is a polynomial, we can evaluate it at negative integers although it is nonsense by definition. 
Stanley proved the following \emph{combinatorial reciprocity theorem}. 
\begin{theorem}[{Stanley \cite[Theorem 1.2]{stanley1973acyclic-dm}}]\label{Stanley CR chromatic polynomial}
Let $ \Gamma $ be a simple graph. 
Then 
$ \chi_{\Gamma}(-n) = (-1)^{|V_{\Gamma}|}\overline{\chi}_{\Gamma}(n)
 $. 
\end{theorem}

We also define a function $ \overline{X}_{\Gamma} $ by 
\begin{align*}
\overline{X}_{\Gamma} &\coloneqq \sum_{(\mathfrak{o}, \kappa)} \prod_{v \in V_{\Gamma}} x_{\kappa(v)}, 
\end{align*}
where the sum ranges over all pairs $ (\mathfrak{o}, \kappa) $ such that $ \mathfrak{o} $ is an acyclic orientation of $ \Gamma $ and $ \kappa \colon V_{\Gamma} \rightarrow \mathbb{Z}_{>0} $ is compatible with $ \mathfrak{o} $. 
By a reason similar to $ X_{\Gamma} $, the function $ \overline{X}_{\Gamma} $ is a generalization of $ \overline{\chi}_{\Gamma} $. 

Let $ \omega $ denote the standard involution of symmetric functions. 
It is characterized as an algebra homomorphism satisfying $ \omega p_{k} = (-1)^{k-1}p_{k} $, where $ p_{k}=\sum_{i=1}^{\infty}x_{i}^{k} $ denotes the power sum symmetric function of degree $ k $ for every positive integer $ k $ (See \cite[Proposition 7.7.5]{stanley1999enumerative}). 
Theorem \ref{Stanley CR chromatic polynomial} is generalized as follows. 
\begin{theorem}[{Stanley \cite[Theorem 4.2]{stanley1995symmetric-aim}}]\label{Stanley CR chromatic symmetric function}
Let $ \Gamma $ be a simple graph. 
Then $ \omega X_{\Gamma} = \overline{X}_{\Gamma} $. 
\end{theorem}

Stanley proved Theorem \ref{Stanley CR chromatic polynomial} and Theorem \ref{Stanley CR chromatic symmetric function} by using the deletion-contraction formula and the reciprocity theorem for $ P $-partitions. 
Beck and Zaslavsky \cite[Corollary 5.5]{beck2006inside-out-aim} showed Theorem \ref{Stanley CR chromatic polynomial} as a corollary of the Ehrhart reciprocity \cite{macdonald1971polynomials-jotlms, mcmullen1978lattice-adm}, which states a relation between the number of lattice points in a lattice polytope and the number of lattice points in the interior of the polytope. 

To see the reason why the Ehrhart reciprocity is effective, we introduce graphic arrangements. 
Let $ \Gamma $ be a simple graph on vertex set $ [\ell] = \{1, \dots, \ell \} $. 
The \textbf{graphic arrangement} $ \mathcal{A}_{\Gamma} $ is a hyperplane arrangement defined by
\begin{align*}
\mathcal{A}_{\Gamma} \coloneqq \Set{ \{z_{i}=z_{j}\} | \{i,j \} \in E_{\Gamma}}, 
\end{align*}
where $ (z_{1}, \dots, z_{\ell}) $ denotes a system of coordinates of the Euclidean space $ \mathbb{R}^{\ell} $. 

Many notions of graphs can be interpreted as notions of graphic arrangements. 
A vertex $ i $ corresponds to the coordinate $ z_{i} $. 
An edge $ \{i,j\} $ corresponds to the hyperplane $ \{z_{i}=z_{j}\} $. 
A direction $ (i,j) $ is considered as the half space $ \{z_{i} > z_{j}\} $. 
Given an acyclic orientation $ \mathfrak{o} $, we consider the intersection $ C \coloneqq \bigcap_{(i,j) \in \mathfrak{o}}\{z_{i} > z_{j}\} $. 
Then $ C $ is a \textbf{chamber} of $ \mathcal{A}_{\Gamma} $, a connected component of the complement $ M(\mathcal{A}_{\Gamma}) \coloneqq \mathbb{R}^{\ell} \setminus \bigcup_{H \in \mathcal{A}_{\Gamma}}H $. 
Furthermore, the construction is a bijection from the acyclic orientations of $ \Gamma $ to the chambers of $ \mathcal{A}_{\Gamma} $
(See \cite[Lemma 7.1]{greene1983interpretation-totams} for details). 
A coloring is considered to be an integer point in $ \mathbb{Z}_{>0}^{\ell} $ and hence a proper coloring corresponds to an integer point in $ M(\mathcal{A}_{\Gamma}) \cap \mathbb{Z}_{>0}^{\ell} $. 
Table \ref{Tab graph vs graphic arrangement} shows a summary. 

\begin{table}[t]
\centering
\begin{tabular}{c|c}
graph & graphic arrangement \\
\hline
vertex $ i $ & coordinate $ z_{i} $ \\
edge $ \{i,j\} $ & hyperplane $ \{z_{i}=z_{j}\} $ \\
direction $ (i,j) $ & half space $ \{z_{i} > z_{j}\} $ \\
acyclic orientation & chamber \\
coloring & integer point \\
proper coloring & integer point in the complement
\end{tabular}
\caption{Graphs versus graphic arrangements}\label{Tab graph vs graphic arrangement}
\end{table}

Let $ \mathfrak{o} $ be an acyclic orientation and $ C $ the corresponding chamber. 
Then a pair $ (\mathfrak{o},\kappa) $ such that $ \kappa \colon [\ell] \rightarrow \mathbb{Z}_{>0} $ is compatible with $ \mathfrak{o} $ corresponds to an integer point in $ \overline{C} \cap \mathbb{Z}_{>0}^{\ell} $, where $ \overline{C} $ denotes the closure of $ C $. 
Since integer points in $ C \cap \mathbb{Z}_{>0}^{\ell} $ correspond to proper colorings, if we apply the Ehrhart reciprocity to polytopes $ \overline{C} \cap [1,n]^{\ell} $ for each chamber $ C $, then we obtain Theorem \ref{Stanley CR chromatic polynomial}, where $ [1,n] $ denotes the closed interval $ \Set{a \in \mathbb{R} | 1 \leq a \leq n} $. 
For details and more general results, see \cite{beck2006inside-out-aim}. 

\begin{remark}\label{csf chamber}
The functions $ X_{\Gamma} $ and $ \overline{X}_{\Gamma} $ for a simple graph $ \Gamma $ on $ [\ell] $ can be expressed in terms of the graphic arrangement $ \mathcal{A}_{\Gamma} $ as follows. 
\begin{align*}
X_{\Gamma} &= \sum_{C}\sum_{\boldsymbol{\alpha} \in C \cap \mathbb{Z}_{>0}^{\ell}} x_{\alpha_{1}} \cdots x_{\alpha_{\ell}}, \\
\overline{X}_{\Gamma} &= \sum_{C}\sum_{\boldsymbol{\alpha} \in \overline{C} \cap \mathbb{Z}_{>0}^{\ell}} x_{\alpha_{1}} \cdots x_{\alpha_{\ell}}, 
\end{align*}
where $ C $ ranges over all chambers of $ \mathcal{A}_{\Gamma} $ and $ \boldsymbol{\alpha} = (\alpha_{1}, \dots, \alpha_{\ell}) $. 
\end{remark}

\subsection{Main results}
In this paper, we will present an analogue of Theorem \ref{Stanley CR chromatic symmetric function} for signed graphs (Theorem \ref{main theorem reciprocity}) and give a result concerning a generalization of Conjecture \ref{Stanley tree conjecture} for some signed paths (Theorem \ref{main signed paths}). 

A \textbf{signed graph} is a quadruple $ \Gamma = (V_{\Gamma}, E_{\Gamma}^{+}, E_{\Gamma}^{-}, L_{\Gamma}) $, where 
\begin{itemize}
\item $ V_{\Gamma} $ is a finite set, whose element is called a \textbf{vertex},  
\item $ E_{\Gamma}^{+} $ consists of 2-element sets in $ V_{\Gamma} $, whose element is called a \textbf{positive edge}, 
\item $ E_{\Gamma}^{-} $ consists of 2-element sets in $ V_{\Gamma} $, whose element is called a \textbf{negative edge}, 
\item $ L_{\Gamma} $ is a subset of $ V_{\Gamma} $, whose element is called a \textbf{loop}. 
\end{itemize}
Let $ E_{\Gamma} $ denote the \textbf{edge set} $ E_{\Gamma} \coloneqq E^{+}_{\Gamma} \sqcup E^{-}_{\Gamma} \sqcup L_{\Gamma} $, where $ \sqcup $ means the disjoint union. 

When $ V_{\Gamma} = [\ell] $, we define the \textbf{signed-graphic arrangement} $ \mathcal{A}_{\Gamma} $ in $ \mathbb{R}^{\ell} $ by
\begin{align*}
\mathcal{A}_{\Gamma} 
\coloneqq \Set{\{z_{i}=z_{j}\} | \{i,j\} \in E_{\Gamma}^{+}} 
\cup \Set{\{z_{i}=-z_{j}\} | \{i,j\} \in E_{\Gamma}^{-}} 
\cup \Set{\{z_{i}=0\} | i \in L_{\Gamma}}. 
\end{align*}

Note that every simple graph $ (V,E) $ can be regarded as a signed graph $ (V,E,\varnothing,\varnothing) $. 
The graphic arrangement $ \mathcal{A}_{(V,E)} $ and the signed-graphic arrangement $ \mathcal{A}_{(V,E,\varnothing,\varnothing)} $ coincide. 

As with simple graphs, it is natural to consider that a coloring of a signed graph is an integer point in the complement $ M(\mathcal{A}_{\Gamma}) = \mathbb{R}^{\ell}\setminus \bigcup_{H \in \mathcal{A}_{\Gamma}}H $ of the signed-graphic arrangement $ \mathcal{A}_{\Gamma} $ and also the edge set $ E_{\Gamma} $ has the matroid structure which stems from the linear dependence matroid on the signed-graphic arrangement. 
Zaslavsky \cite{zaslavsky1982signed-dm, zaslavsky1989biased-joctsb, zaslavsky1991biased-joctsb, zaslavsky1995biased-joctsb, zaslavsky2001supersolvable-ejoc, zaslavsky2003biased-joctsb} studied the matroids and the chromatic polynomials of signed graphs and gain graphs (further generalizations of graphs whose edges are labeled by group elements). 

We introduce analogues for a signed graph $ \Gamma $ following Remark \ref{csf chamber}. 
\begin{definition}
Let $ \Gamma $ be a signed graph with vertex set $ [\ell] $. 
Define $ X_{\Gamma} $ and $ \overline{X}_{\Gamma} $ by 
\begin{align*}
X_{\Gamma} &\coloneqq \sum_{C}\sum_{\boldsymbol{\alpha} \in C \cap \mathbb{Z}^{\ell}} x_{\alpha_{1}} \cdots x_{\alpha_{\ell}}, \\
\overline{X}_{\Gamma} &\coloneqq \sum_{C}\sum_{\boldsymbol{\alpha} \in \overline{C} \cap \mathbb{Z}^{\ell}} x_{\alpha_{1}} \cdots x_{\alpha_{\ell}}, 
\end{align*}
where $ C $ ranges over all chambers of $ \mathcal{A}_{\Gamma} $, $ \boldsymbol{\alpha} = (\alpha_{1}, \dots, \alpha_{\ell}) $, and $ (x_{i})_{i \in \mathbb{Z}} $ denote countably infinite indeterminates indexed by integers. 
We call $ X_{\Gamma} $ the \textbf{chromatic signed-symmetric function} of $ \Gamma $. 
\end{definition}

Every permutation $ \sigma $ of $ \mathbb{Z} $ acts the indeterminates $ (x_{i})_{i \in \mathbb{Z}} $ by $ \sigma(x_{i}) \coloneqq x_{\sigma(i)} $. 
A permutation $ \sigma $ of $ \mathbb{Z} $ is called a \textbf{signed permutation} if $ \sigma(-i) = -\sigma(i) $ for all $ i \in \mathbb{Z} $. 
A degree-bounded function which is invariant under the action of signed permutations is called a \textbf{signed-symmetric} function (See Definition \ref{definition of signed-symmetric functions}). 
It is easy to see that $ X_{\Gamma} $ is a signed-symmetric function (Proposition \ref{X is signed-symmetric}). 

Note that Egge \cite{egge2016chromatic} introduced another analogous invariant under the action of signed permutations for a loopless signed graph. 
Chmutov et al. \cite{chmutov2020b-symmetric} also investigate the invariant.

For nonnegative integers $ a $ and $ b $ with $ (a,b) \neq (0,0) $, we define
\begin{align*}
p_{\begin{psmallmatrix} a \\ b \end{psmallmatrix}}
\coloneqq \sum_{i \in \mathbb{Z}}x_{i}^{a}x_{-i}^{b}, 
\end{align*}
which is an analogue of the power sum symmetric function and is invariant under the action of signed permutations.

Signed-symmetric functions form a ring, which is similar to symmetric functions. 
Furthermore, the ring is actually a free commutative algebra (Theorem \ref{power sum basis}) and we may define an involution $ \omega $ on the ring by
\begin{align*}
\omega x_{0} \coloneqq -x_{0} \quad \text{ and } \quad
\omega p_{\begin{psmallmatrix} a \\ b \end{psmallmatrix}} \coloneqq (-1)^{a+b-1} p_{\begin{psmallmatrix} a \\ b \end{psmallmatrix}}. 
\end{align*}

One of the main results of this article is the following combinatorial reciprocity theorem generalizing Theorem \ref{Stanley CR chromatic symmetric function}. 
\begin{theorem}[\textcolor{red}{This was proved by Wolfgang \cite{wolfgang1997two}}]\label{main theorem reciprocity}
Let $ \Gamma $ be a signed graph. 
Then $ \omega X_{\Gamma} = \overline{X}_{\Gamma} $. 
\end{theorem}


In order to state the other result, we introduce signed trees and paths. 
A \textbf{signed tree} (resp. \textbf{signed path}) is a signed graph $ T = (V_{T}, E_{T}^{+}, E_{T}^{-}, \varnothing) $ such that $ E_{T}^{+} \cap E_{T}^{-} = \varnothing $ and the simple graph $ (V_{T}, E_{T}^{+}\cup E_{T}^{-}) $ is a tree (resp. path). 

We may generalize Conjecture \ref{Stanley tree conjecture} as follows. 

\begin{problem}\label{question signed path}
Does the chromatic signed-symmetric function distinguish signed trees? 
Namely, if $ T_{1} $ and $ T_{2} $ are signed trees with $ X_{T_{1}} = X_{T_{2}} $, then are $ T_{1} $ and $ T_{2} $ isomorphic? 
\end{problem}

Even in the case of signed paths, the problem is hard. 
We have checked the chromatic signed-symmetric functions of signed paths with a computer. 
The question in Problem \ref{question signed path} is affirmative for signed paths with up to $ 15 $ vertices. 
 
Let $ \boldsymbol{\alpha} = (\alpha_{1}, \dots, \alpha_{\ell}) \in \mathbb{Z}_{>0}^{\ell} $ denote an integer composition of length $ \ell $. 
Let $ P_{\boldsymbol{\alpha}} $ be the signed path obtained by connecting paths consisting of positive edges $ P_{\alpha_{1}}, \dots, P_{\alpha_{\ell}} $ with negative edges in this order. 
Note that every signed path is isomorphic to $ P_{\boldsymbol{\alpha}} $ for some $ \boldsymbol{\alpha} $. 
We say that a composition $ \boldsymbol{\alpha} $ is \textbf{unimodal} if there exists $ t \in \{1, \dots, \ell\} $ such that $ \alpha_{1} \leq \dots \leq \alpha_{t} \geq \dots \geq \alpha_{\ell} $. 


\begin{theorem}\label{main signed paths}
Suppose that $ X_{P_{\boldsymbol{\alpha}}} = X_{P_{\boldsymbol{\beta}}} $. 
If one of the following conditions holds, then $ P_{\boldsymbol{\alpha}} $ and $ P_{\boldsymbol{\beta}} $ are isomorphic. 
\begin{enumerate}[(1)]
\item The length of $ \boldsymbol{\alpha} $ is less than or equal to $ 4 $. 
\item $ \boldsymbol{\alpha} $ and $ \boldsymbol{\beta} $ are unimodal. 
\end{enumerate}
\end{theorem}

The organization of this article is as follows. 
In \S \ref{sec: preliminaries}, we introduce signed-symmetric functions and prove that the ring of signed-symmetric functions is a free commutative algebra generated by $ x_{0} $ and the power sum signed-symmetric functions $ p_{\begin{psmallmatrix}
a \\ b
\end{psmallmatrix}} $ (Theorem \ref{power sum basis}). 
We also review some notions for arrangements and signed-graphs
Especially we prove the key theorems of this article (Theorem \ref{XA} and \ref{oXA}), which are essentially showed by Beck and Zaslavsky \cite{beck2006inside-out-aim}. 

In \S \ref{sec:chromatic signed-symmetric functions}, we study the chromatic signed-symmetric functions. 
An important property of chromatic signed-symmetric functions is the power sum expansion (Theorem \ref{power sum expansion subset} and \ref{power sum expansion flat}). 
Using the power sum expansion, we prove the main theorems.

\section{Preliminaries}\label{sec: preliminaries}

%
%

\subsection{Ring of signed-symmetric functions}\label{Subsec ho functions}
Let $ \mathbb{Z}\llbracket \boldsymbol{x}\rrbracket $ denote the ring of formal power series over $ \mathbb{Z} $ in the indeterminates $ \boldsymbol{x} = (x_{i})_{i \in \mathbb{Z}} $. 
\begin{definition}\label{definition of signed-symmetric functions}
A formal power series $ f \in \mathbb{Z}\llbracket\boldsymbol{x}\rrbracket $ is called a \textbf{signed-symmetric function} if the following conditions are satisfied. 
\begin{enumerate}[(i)]
\item The degrees of the monomials in $ f $ are bounded. 
\item $ f $ is invariant under the action of signed permutations, that is, the action defined by $ \sigma(x_{i}) = x_{\sigma(i)} $ with the permutations $ \sigma $ on $ \mathbb{Z} $ such that $ \sigma(-i) = -\sigma(i) $ for all $ i \in \mathbb{Z} $. 
\end{enumerate}
The signed-symmetric functions form a subring of $ \mathbb{Z}\llbracket \boldsymbol{x} \rrbracket $. 
We call it the \textbf{ring of signed-symmetric functions} over $ \mathbb{Z} $ , denoted by $ \SSym_{\mathbb{Z}} $. 
Moreover, let $ \SSym_{\mathbb{Q}} \coloneqq \SSym_{\mathbb{Z}} \otimes_{\mathbb{Z}} \mathbb{Q} $. 
\end{definition}

Recall that every well-known basis for the ring of symmetric functions, including the monomial basis or the power sum basis, is indexed by integer partitions. 
In order to introduce bases for $ \SSym_{\mathbb{Q}} $, we will define an analogous notion of integer partitions as follows. 

Consider $ 2 \times r $ rectangular arrays whose entries are nonnegative integers such that every column contains a positive integer. 
We identify two such arrays if one is obtained from the other by permuting columns and interchanging entries in some columns. 
For example, 
\begin{align*}
\begin{pmatrix}
1 \\ 2
\end{pmatrix} = \begin{pmatrix}
2 \\ 1
\end{pmatrix} \quad \text{ and } \quad
\begin{pmatrix}
2 & 1 & 1 \\
0 & 1 & 0
\end{pmatrix} = \begin{pmatrix}
1 & 1 & 0  \\
1 & 0 & 2 
\end{pmatrix}. 
\end{align*}
We will use the symbols $ \lambda $ and $ \mu $ for the arrays. 
A \textbf{partition} is a pair $ (u, \lambda) $, where $ u $ is a nonnegative integer, which is considered to be an analogue of a usual integer partition. 
When $ u=0 $, we will identify $ (0,\lambda) $ with $ \lambda $. 

Every monomial in $ \mathbb{Z}\llbracket \boldsymbol{x} \rrbracket $ with coefficient $ 1 $ is of the form 
\begin{align*}
x_{0}^{u}x_{i_{1}}^{a_{1}}x_{-i_{1}}^{b_{1}} x_{i_{2}}^{a_{2}}x_{-i_{2}}^{b_{2}} \cdots x_{i_{r}}^{a_{r}}x_{-i_{r}}^{b_{r}}, 
\end{align*}
where $ r \geq 0 $, at least one of $ a_{i_{j}} $ and $ b_{i_{j}} $ is positive for each $ j \in \{1,\dots, r\} $, and $ \{\pm i_{j}\} \cap \{\pm i_{k} \} = \varnothing $ if $ j \neq k $. 
Then the pair $ (u,\lambda) $ is called the \textbf{type} of the monomial, where $ \lambda $ denotes the array
\begin{align*}
\begin{pmatrix}
a_{1} & a_{2} & \dots & a_{r} \\
b_{1} & b_{2} & \dots & b_{r}
\end{pmatrix}. 
\end{align*}

\begin{definition}
The \textbf{monomial signed-symmetric function} $ m_{(u,\lambda)} $ is the sum of all the monomials with coefficient $ 1 $ of type $ (u,\lambda) $. 
\end{definition}

\begin{proposition}\label{monomial basis}
The set $ \{m_{(u,\lambda)}\}_{(u,\lambda)} $ is a basis for $ \SSym_{\mathbb{Z}} $ as a module. 
\end{proposition}
\begin{proof}
Since the type of a monomial is invariant under the action of signed-permutations, $ m_{(u,\lambda)} $ is a signed-symmetric function. 
Moreover, every signed-symmetric function $ f $ is expressed as a linear combination of finitely many monomial signed-symmetric functions since the degrees of monomials in $ f $ is bounded. 
Clearly, the expression is unique. 
\end{proof}

\begin{remark}
Note that Egge \cite{egge2016chromatic} studied the subring of $ \SSym_{\mathbb{Q}} $ which consists of all the signed-symmetric functions not containing the indeterminate $ x_{0} $ and mentioned that the monomial signed-symmetric functions $ m_{\lambda} = m_{(0,\lambda)} $ form a basis for the ring. 
Also, note that $ m_{(u,\lambda)} = x_{0}^{u}m_{\lambda} $. 
\end{remark}

\begin{definition}
Define the \textbf{power sum signed-symmetric function} $ p_{(u,\lambda)} $ by 
\begin{align*}
p_{(u,\lambda)} \coloneqq x_{0}^{u} p_{\begin{psmallmatrix} a_{1} \\ b_{1} \end{psmallmatrix}} 
p_{\begin{psmallmatrix} a_{2} \\ b_{2} \end{psmallmatrix}} \cdots p_{\begin{psmallmatrix} a_{r} \\ b_{r} \end{psmallmatrix}}, 
\end{align*}
where 
\begin{align*}
\lambda = \begin{pmatrix}
a_{1} & a_{2} & \cdots & a_{r} \\
b_{1} & b_{2} & \cdots & b_{r}
\end{pmatrix} \quad \text{ and } \quad 
p_{\begin{psmallmatrix} a \\ b \end{psmallmatrix}} = \sum_{i \in \mathbb{Z}}x_{i}^{a}x_{-i}^{b}. 
\end{align*}
\end{definition}

\begin{theorem}\label{power sum basis}
The set $ \{p_{(u,\lambda)}\}_{(u,\lambda)} $ is a basis for $ \SSym_{\mathbb{Q}} $ as a vector space over $ \mathbb{Q} $. 
Especially, $ \SSym_{\mathbb{Q}} $ is a free commutative algebra generated by $ x_{0} $ and the power sum signed-symmetric functions $ p_{\begin{psmallmatrix} a \\ b \end{psmallmatrix}} $ with $ a \geq b $ over $ \mathbb{Q} $. 
\end{theorem}
\begin{proof}
First, we define a partial order on the partitions. 
We say that $ \lambda $ covers $ \mu $ if $ \lambda $ can be obtained by summing two column vectors in a representative of $ \mu $. 
See Figure \ref{Fig Hasse} for an example. 
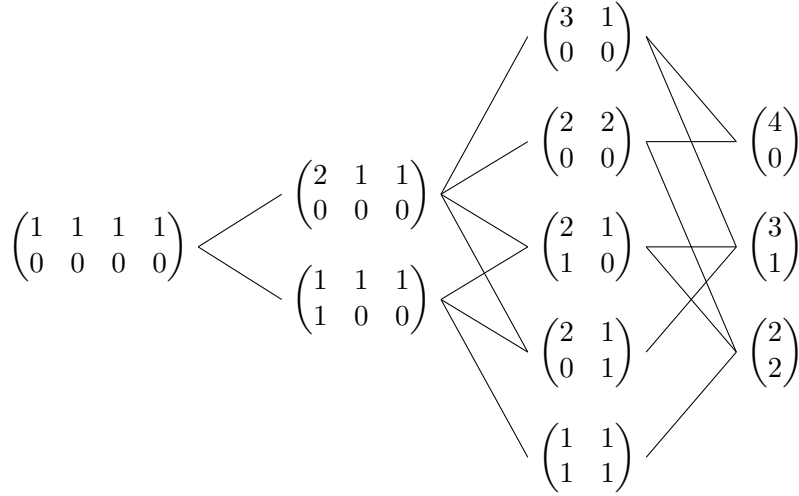
\begin{figure}[t]
\centering
\begin{tikzpicture}
\draw (0,0) node (0a) {$ \begin{pmatrix}
1 & 1 & 1 & 1 \\
0 & 0 & 0 & 0
\end{pmatrix} $}; 
\draw (3.5,0.7) node (1a) {$ \begin{pmatrix}
2 & 1 & 1  \\
0 & 0 & 0 
\end{pmatrix} $}; 
\draw (3.5,-0.7) node (1b) {$ \begin{pmatrix}
1 & 1 & 1  \\
1 & 0 & 0 
\end{pmatrix} $}; 
\draw (6.5,2.8) node (2a) {$ \begin{pmatrix}
3 & 1  \\
0 & 0 
\end{pmatrix} $}; 
\draw (6.5,1.4) node (2b) {$ \begin{pmatrix}
2 & 2  \\
0 & 0 
\end{pmatrix} $}; 
\draw (6.5,0) node (2c) {$ \begin{pmatrix}
2 & 1  \\
1 & 0 
\end{pmatrix} $}; 
\draw (6.5,-1.4) node (2d) {$ \begin{pmatrix}
2 & 1  \\
0 & 1 
\end{pmatrix} $}; 
\draw (6.5,-2.8) node (2e) {$ \begin{pmatrix}
1 & 1  \\
1 & 1 
\end{pmatrix} $}; 
\draw (9,1.4) node (3a) {$ \begin{pmatrix}
4   \\
0  
\end{pmatrix} $}; 
\draw (9,0) node (3b) {$ \begin{pmatrix}
3   \\
1  
\end{pmatrix} $}; 
\draw (9,-1.4) node (3c) {$ \begin{pmatrix}
2   \\
2  
\end{pmatrix} $}; 
\draw (0a.east)--(1a.west);
\draw (0a.east)--(1b.west);
\draw (1a.east)--(2a.west);
\draw (1a.east)--(2b.west);
\draw (1a.east)--(2c.west);
\draw (1a.east)--(2d.west);
\draw (1b.east)--(2c.west);
\draw (1b.east)--(2d.west);
\draw (1b.east)--(2e.west);
\draw (2a.east)--(3a.west);
\draw (2a.east)--(3b.west);
\draw (2b.east)--(3a.west);
\draw (2b.east)--(3c.west);
\draw (2c.east)--(3b.west);
\draw (2c.east)--(3c.west);
\draw (2d.east)--(3b.west);
\draw (2e.east)--(3c.west);
\end{tikzpicture}
\caption{Hasse diagram}
\label{Fig Hasse}
\end{figure}
Furthermore, we define $ (u,\lambda) < (v,\mu) $ if $ u < v $ or both of $ u=v $ and $ \lambda < \mu $ hold. 

One can deduce that there exist nonnegative integers $ R_{(u,\lambda), (v,\mu)} $ such that 
\begin{align*}
p_{(u,\lambda)} = \sum_{(v,\mu) \geq (u,\lambda)}R_{(u,\lambda), (v,\mu)}m_{(v,\mu)} \quad \text{ and } \quad R_{(u,\lambda), (v,\lambda)} > 0. 
\end{align*}
By Proposition \ref{monomial basis}, $ \{m_{(u,\lambda)}\}_{(u,\lambda)} $ is a basis for $ \SSym_{\mathbb{Q}} $ and hence the equalities show that $ \{p_{(u,\lambda)}\}_{(u,\lambda)} $ is also a basis for $ \SSym_{\mathbb{Q}} $. 
\end{proof}

\begin{problem}
Is there a combinatorial interpretation of the coefficients $ R_{(u,\lambda), (v,\lambda)} $? 
See \cite[Proposition 7.7.1]{stanley1999enumerative} for a similar result about symmetric functions. 
\end{problem}

\begin{remark}
Since 
\begin{align*}
m_{\begin{psmallmatrix}
1 & 1 \\
0 & 0
\end{psmallmatrix}} = \dfrac{1}{2}p_{\begin{psmallmatrix}
1 & 1 \\
0 & 0
\end{psmallmatrix}} -\dfrac{1}{2}p_{\begin{psmallmatrix}
2 \\
0
\end{psmallmatrix}}-x_{0}p_{\begin{psmallmatrix}
1 \\ 
0
\end{psmallmatrix}}+x_{0}^{2}, 
\end{align*}
the set $ \{p_{(u,\lambda)}\}_{(u,\lambda)} $ is \emph{not} a basis for $ \SSym_{\mathbb{Z}} $, which is analogous to the case of symmetric functions. 
\end{remark}
\begin{problem}
Is $ \SSym_{\mathbb{Z}} $ a free commutative ring? 
\end{problem}

Thanks to Theorem \ref{power sum basis}, we may define the involution $ \omega $ as follows. 
\begin{definition}
Define an involution $ \omega \colon \SSym_{\mathbb{Q}} \rightarrow \SSym_{\mathbb{Q}} $ by 
\begin{align*}
\omega x_{0} \coloneqq -x_{0} \quad \text{ and } \quad \omega p_{\begin{psmallmatrix} a \\ b \end{psmallmatrix}} \coloneqq (-1)^{a+b-1} p_{\begin{psmallmatrix} a \\ b \end{psmallmatrix}}, 
\end{align*}
or equivalently 
\begin{align*}
\omega p_{(u,\lambda)} = (-1)^{u+\sum_{i=1}^{r}(a_{i}+b_{i})-r}p_{(u,\lambda)}, \quad \text{ where } \quad \lambda = \begin{pmatrix}
a_{1} & a_{2} & \cdots & a_{r} \\
b_{1} & b_{2} & \cdots & b_{r}
\end{pmatrix}. 
\end{align*}
\end{definition}

Let $ \Sym_{\mathbb{Q}} $ denote the ring of ordinary symmetric functions over $ \mathbb{Q} $. 
This ring is generated by the power sum symmetric functions $ p_{k} $, which are algebraically independent over $ \mathbb{Q} $. 
Hence we may define the embedding $ \iota \colon \Sym_{\mathbb{Q}} \to \SSym_{\mathbb{Q}} $ by
\begin{align*}
\iota(p_{k}) \coloneqq p_{\begin{psmallmatrix}
k \\ 0
\end{psmallmatrix}}. 
\end{align*}

We can obtain a symmetric function from a signed-symmetric function by substituting $ x_{i} = 0 $ for $ i \leq 0 $ and remaining $ x_{i} $ for $ i > 0 $ since a permutation $ \sigma $ on the positive integers can be extended to a signed-symmetric permutation by $ \sigma(0)\coloneqq0 $ and $ \sigma(i) \coloneqq -\sigma(-i) $ for $ i < 0 $. 
Let $ \pi \colon \SSym_{\mathbb{Q}} \rightarrow \Sym_{\mathbb{Q}} $ denote this assignment. 
The map $ \pi $ is a surjective homomorphism and characterized by
\begin{align*}
\pi(x_{0}) = 0, \qquad
\pi\left(p_{\begin{psmallmatrix}
a \\ b
\end{psmallmatrix}}\right) = \begin{cases}
p_{a} & (a>0, b=0) \\
0 & (a \geq b > 0)
\end{cases}
\end{align*}

The following proposition is obvious. 
\begin{proposition}\label{injection surjection}
The maps $ \iota \colon \Sym_{\mathbb{Q}} \to \SSym_{\mathbb{Q}} $ and $ \pi \colon \SSym_{\mathbb{Q}} \rightarrow \Sym_{\mathbb{Q}} $ are homomorphisms of involutive rings. 
Namely, $ \iota(\omega f) = \omega \iota(f) $ and $ \pi(\omega f) = \omega \pi(f) $. 
\end{proposition}

\subsection{Hyperplane arrangements}

A \textbf{(central) hyperplane arrangement} $ \mathcal{A} $ is a collection of finitely many vector subspaces of codimension $ 1 $ in the Euclidean space $ \mathbb{R}^{\ell} $. 
The \textbf{intersection lattice} $ L(\mathcal{A}) $ is defined by
\begin{align*}
L(\mathcal{A}) \coloneqq \Set{\bigcap_{H \in \mathcal{B}}H | \mathcal{B} \subseteq \mathcal{A}}
\end{align*}
with the reverse inclusion order: $ Y \leq Z \Leftrightarrow Y \supseteq Z $. 
Note that the ambient space $ \mathbb{R}^{\ell} $, considered as the empty intersection, is the minimal element of $ L(\mathcal{A}) $, denoted by $ \hat{0} $. 
It is well known that $ L(\mathcal{A}) $ is a geometric lattice. 

Let $ Y \in L(\mathcal{A}) $. 
The \textbf{localization} $ \mathcal{A}_{Y} $ and the \textbf{restriction} $ \mathcal{A}^{Y} $ is defined by 
\begin{align*}
\mathcal{A}_{Y} &\coloneqq \Set{H \in \mathcal{A} | H \supseteq Y}, \\
\mathcal{A}^{Y} &\coloneqq \Set{H \cap Y | H \in \mathcal{A}\setminus \mathcal{A}_{Y}}. 
\end{align*}
Note that the restriction $ \mathcal{A}^{Y} $ is an arrangement in the vector space $ Y $. 

We define the \textbf{characteristic polynomial} $ \chi_{\mathcal{A}}(t) $ by 
\begin{align*}
\chi_{\mathcal{A}}(t) \coloneqq \sum_{Y \in L(\mathcal{A})}\mu(\hat{0},Y)t^{\dim Y}, 
\end{align*}
where $ \mu $ denotes the M\"{o}bius function on the lattice $ L(\mathcal{A}) $, which is defined recursively by 
\begin{align*}
\mu(Y,Z) \coloneqq \begin{cases}
0 & \text{ if } Y \not\leq Z, \\
1 & \text{ if } Y = Z, \\
\displaystyle -\sum_{Y \leq W < Z} \mu(Y,W) & \text{ if } Y<Z.
\end{cases} 
\end{align*}

Recall that $ M(\mathcal{A}) $ denotes the complement of $ \mathcal{A} $, that is, $ M(\mathcal{A}) = \mathbb{R}^{\ell}\setminus \bigcup_{H \in \mathcal{A}} H $ and call a connected component of $ M(\mathcal{A}) $ a chamber. 

\begin{definition}
Define formal power series $ X_{\mathcal{A}} $ and $ \overline{X}_{\mathcal{A}} $ by 
\begin{align*}
X_{\mathcal{A}} \coloneqq \sum_{C}\sum_{\boldsymbol{\alpha} \in C \cap \mathbb{Z}^{\ell}} x_{\boldsymbol{\alpha}} \quad \text{ and } \quad 
\overline{X}_{\mathcal{A}} \coloneqq \sum_{C}\sum_{\boldsymbol{\alpha} \in \overline{C} \cap \mathbb{Z}^{\ell}} x_{\boldsymbol{\alpha}}, 
\end{align*}
where $ C $ ranges over all chambers of $ \mathcal{A} $, $ \boldsymbol{\alpha} = (\alpha_{1}, \dots, \alpha_{\ell}) $, and $ x_{\boldsymbol{\alpha}} \coloneqq x_{\alpha_{1}} \cdots x_{\alpha_{\ell}} $. 
\end{definition}

\begin{theorem}[See also Beck--Zaslavsky {\cite[Theorem 3.1 (3.1)]{beck2006inside-out-aim}}]\label{XA}
\begin{align*}
X_{\mathcal{A}} = \sum_{Y \in L(\mathcal{A})} \mu(\hat{0}, Y) \sum_{\boldsymbol{\alpha} \in Y \cap \mathbb{Z}^{\ell}} x_{\boldsymbol{\alpha}}. 
\end{align*}
\end{theorem}
\begin{proof}
For every $ W \in L(\mathcal{A}) $, we have
\begin{align*}
\sum_{\boldsymbol{\alpha} \in W \cap \mathbb{Z}^{\ell}}x_{\boldsymbol{\alpha}} = \sum_{Y \geq W}X_{\mathcal{A}^{Y}}. 
\end{align*}
By the M\"{o}bius inversion formula, 
\begin{align*}
X_{\mathcal{A}^{W}} = \sum_{Y \geq W}\mu(W,Y)\sum_{\boldsymbol{\alpha} \in Y \cap \mathbb{Z}^{\ell}}x_{\boldsymbol{\alpha}}. 
\end{align*}
Putting $ W = \hat{0} $ yields the desired result. 
\end{proof}

\begin{lemma}[Beck--Zaslavsky {\cite[Lemma 3.4]{beck2006inside-out-aim}}]\label{BZ multiplicity}
For $ \boldsymbol{\alpha} \in \mathbb{R}^{\ell} $, let $ m_{\mathcal{A}}(\boldsymbol{\alpha}) $ denote the number of chambers $ C $ of $ \mathcal{A} $ such that $ \boldsymbol{\alpha} \in \overline{C} $. 
Then $ m_{\mathcal{A}}(\boldsymbol{\alpha}) $ is equal to the number of chambers of the localization $ \mathcal{A}_{Z_{\boldsymbol{\alpha}}} $, where $ Z_{\boldsymbol{\alpha}} \in L(\mathcal{A}) $ denotes the minimal element containing $ \boldsymbol{\alpha} $. 
Furthermore, $ m_{\mathcal{A}}(\boldsymbol{\alpha}) = |\chi_{\mathcal{A}_{Z_{\boldsymbol{\alpha}}}}(-1)| $. 
\end{lemma}

\begin{lemma}[Rota {\cite[Theorem 4]{rota1964foundations-zfwuvg}}]\label{Rota sign theorem}
If $ Y \in L(\mathcal{A}) $ covers $ Z \in L(\mathcal{A}) $, then $ \mu(\hat{0},Y) $ and $ \mu(\hat{0}, Z) $ have opposite signs. 
In particular, $ |\chi_{\mathcal{A}}(-1)| = \sum_{Y \in L(\mathcal{A})} \left| \mu(\hat{0}, Y) \right| $. 
\end{lemma}

\begin{theorem}[See also Beck--Zaslavsky {\cite[Theorem 3.1 (3.2)]{beck2006inside-out-aim}}]\label{oXA}
\begin{align*}
\overline{X}_{\mathcal{A}} = \sum_{Y \in L(\mathcal{A})} \left|\mu(\hat{0}, Y)\right| \sum_{\boldsymbol{\alpha} \in Y \cap \mathbb{Z}^{\ell}} x_{\boldsymbol{\alpha}}. 
\end{align*}
\end{theorem}
\begin{proof}
By Lemma \ref{BZ multiplicity} and \ref{Rota sign theorem}, 
\begin{align*}
\overline{X}_{\mathcal{A}} 
= \sum_{\boldsymbol{\alpha} \in \mathbb{Z}^{\ell}} m_{\mathcal{A}}(\boldsymbol{\alpha})x_{\boldsymbol{\alpha}}
= \sum_{\boldsymbol{\alpha} \in \mathbb{Z}^{\ell}} \sum_{Y \leq Z_{\boldsymbol{\alpha}}} \left| \mu(\hat{0}, Y) \right|x_{\boldsymbol{\alpha}}
= \sum_{Y \in L(\mathcal{A})} \left| \mu(\hat{0}, Y) \right|\sum_{\substack{\boldsymbol{\alpha} \in \mathbb{Z}^{\ell} \\ Z_{\boldsymbol{\alpha}}\geq Y}}x_{\boldsymbol{\alpha}}. 
\end{align*}
One can prove that 
\begin{align*}
\Set{\boldsymbol{\alpha} \in \mathbb{Z}^{\ell} | Z_{\boldsymbol{\alpha}} \subseteq Y} = Y \cap \mathbb{Z}^{\ell}. 
\end{align*}
Thus the assertion holds. 
\end{proof}

\subsection{Signed graphs}
A \textbf{path} on distinct vertices $ v_{1}, \dots, v_{k} $ of a signed graph $ \Gamma $ is a subset of $ E^{+}_{\Gamma} \sqcup E^{-}_{\Gamma} $ consisting of edges $ \{v_{1}, v_{2}\}, \{v_{2}, v_{3}\}, \dots, \{v_{k-1}, v_{k}\} $. 
A signed graph is called \textbf{connected} if every pair of vertices can be joined by a path. 
Every signed graph $ \Gamma $ can be decomposed into the connected components in the usual way. 

For $ k \geq 3 $, a \textbf{cycle of length $ k $} is  a path above with an edge $ \{v_{k}, v_{1}\} $. 
A pair of the positive and negative edges between two vertices is called a \textbf{cycle of length $ 2 $}. 
A loop is also considered as a \textbf{cycle of length $ 1 $}. 

A cycle of length $ k \geq 2 $ is called \textbf{balanced} if the number of negative edges of it is even. 
Otherwise, call the cycle \textbf{unbalanced}. 
Thus every cycle of length $ 2 $ is unbalanced. 
Moreover, a loop is defined to be \textbf{unbalanced}. 
A subset of $ E_{\Gamma} $ is \textbf{balanced} if every cycle in it is balanced.  
Otherwise, call it \textbf{unbalanced}. 

A \textbf{tight handcuff} is the union of two unbalanced cycles sharing exactly one vertex. 
A \textbf{loose handcuff} is the union of two unbalanced cycles and a path such that the cycles share no vertices, an endvertex of the path belongs to one of the cycles, the other endvertex belongs to the other cycle, and the internal vertices of the path do not belong to the cycles. 

A subset of $ E_{\Gamma} $ is called a \textbf{circuit} if it is a balanced cycle, a tight handcuff, or a loose handcuff. 
The set of circuits defines a matroid on $ E_{\Gamma} $, which is called the \textbf{frame matroid} (or the \textbf{bias matroid}) of $ \Gamma $. 
Then a subset $ S \subseteq E_{\Gamma} $ is a \textbf{flat} if and only if 
\begin{align*}
\Set{ e \in E_{\Gamma}\setminus S | \text{ $ e $ and some elements in $ S $ form a circuit}} = \varnothing. 
\end{align*}
Let $ L(\Gamma) $ denote the set of flats of $ \Gamma $. 
It is a geometric lattice with the inclusion order. 

A subset of $ E_{\Gamma} $ yields an element of $ L(\mathcal{A}_{\Gamma}) $ by taking the intersection of the hyperplanes corresponding to the edges. 
This assignment leads to the following isomorphism. 
\begin{proposition}[Zaslavsky {\cite[Theorem 2.1(a)]{zaslavsky2003biased-joctsb}}]\label{Zaslavsky lattice isomorphism}
Let $ \Gamma $ be a signed graph. 
Then $ L(\Gamma) \simeq L(\mathcal{A}_{\Gamma}) $. 
\end{proposition}

A \textbf{coloring} $ \kappa $ of a signed graph $ \Gamma $ is a map $ \kappa \colon V_{\Gamma} \to \mathbb{Z} $. 
A coloring $ \kappa $ is called \textbf{zero-free} if $ \kappa(v) \neq 0 $ for any $ v \in V_{\Gamma} $. 
A coloring $ \kappa $ is called \textbf{proper} if the following conditions hold. 
\begin{enumerate}[(1)]
\item $ \{v,w\} \in E_{\Gamma}^{+} \Rightarrow \kappa(v) \neq \kappa(w) $. 
\item $ \{v,w\} \in E_{\Gamma}^{-} \Rightarrow \kappa(v) \neq -\kappa(w) $. 
\item $ v \in L_{\Gamma} \Rightarrow \kappa(v) \neq 0 $. 
\end{enumerate}

Let $ [\pm n] \coloneqq \{0, \pm 1, \dots, \pm n\} $. 
An \textbf{$ n $-coloring} of $ \Gamma $ is a map $ \kappa \colon V_{\Gamma} \to [\pm n] $. 
Zaslavsky \cite{zaslavsky1982signed-dm} introduced two kinds of chromatic polynomials for a signed graph. 
\begin{theorem}[Zaslavsky {\cite[Theorem 2.2]{zaslavsky1982signed-dm}}]
Given a signed graph $ \Gamma $, there exist monic polynomials $ \chi_{\Gamma}(t), \chi^{\ast}_{\Gamma}(t) \in \mathbb{Z}[t] $ such that $ \chi_{\Gamma}(2n+1) $ coincides with the number of proper $ n $-colorings of $ \Gamma $ and $ \chi_{\Gamma}(2n) $ coincides with the number of proper zero-free $ n $-colorings of $ \Gamma $. 
We call $ \chi_{\Gamma}(t) $ the \textbf{chromatic polynomial} of $ \Gamma $ and $ \chi^{\ast}_{\Gamma}(t) $ the \textbf{zero-free chromatic polynomial} of $ \Gamma $. 
\end{theorem}

\section{Chromatic signed-symmetric functions}\label{sec:chromatic signed-symmetric functions}

\subsection{Basic properties}
Since a proper coloring corresponds to an integer point of the complement of the signed-graphic arrangement, we may define the chromatic signed-symmetric function of a signed graph $ \Gamma $ by
\begin{align*}
X_{\Gamma} = \sum_{\substack{\kappa \colon V_{\Gamma} \to \mathbb{Z} \\ \kappa \text{ is proper}}} x_{\kappa},  
\end{align*}
where $ x_{\kappa} \coloneqq \prod_{v \in V_{\Gamma}}x_{\kappa(v)} $. 

\begin{proposition}\label{X is signed-symmetric}
Let $ \Gamma $ be a signed graph. 
Then $ X_{\Gamma} \in \SSym_{\mathbb{Z}} $. 
\end{proposition}
\begin{proof}
Since every monomial in $ X_{\Gamma} $ is of degree $ |V_{\Gamma}| $, it suffices to show that $ X_{\Gamma} $ is invariant under the action of signed permutations. 
Let $ \kappa \colon V_{\Gamma} \to \mathbb{Z} $ be a coloring and $ \sigma $ a signed permutation. 

It is easy to show the following. 
\begin{itemize}
\item $ \kappa(v) = \kappa(w) $ if and only if $ (\sigma \circ \kappa)(v) = (\sigma \circ \kappa)(w) $, 
\item $ \kappa(v) = -\kappa(w) $ if and only if $ (\sigma \circ \kappa)(v) = -(\sigma\circ \kappa)(w) $, 
\item $ \kappa(v) = 0 $ if and only if $ (\sigma \circ \kappa)(v) = 0 $. 
\end{itemize}

Therefore $ \kappa $ is proper if and only if $ \sigma \circ \kappa $ is proper. 
Hence if the monomial $ x_{\kappa} $ appears in $ X_{\Gamma} $, then so does $ x_{\sigma\circ\kappa} $. 
Thus $ X_{\Gamma} $ is invariant under the action of signed permutations. 
\end{proof}

Define the \textbf{disjoint union} of signed graphs $ \Gamma_{1} $ and $ \Gamma_{2} $ by
\begin{align*}
\Gamma_{1} \sqcup \Gamma_{2} \coloneqq \left(V_{\Gamma_{1}}\sqcup V_{\Gamma_{2}}, E_{\Gamma_{1}}^{+}\sqcup E_{\Gamma_{2}}^{+}, E_{\Gamma_{1}}^{-}\sqcup E_{\Gamma_{2}}^{-}, L_{\Gamma_{1}}\sqcup L_{\Gamma_{2}} \right). 
\end{align*}
The following proposition is obvious. 

\begin{proposition}\label{disjoint union product}
$ X_{\Gamma_{1}\sqcup \Gamma_{2}} = X_{\Gamma_{1}} X_{\Gamma_{2}} $. 
\end{proposition}

The following proposition shows that the chromatic signed-symmetric function unifies two kinds of chromatic polynomials for a signed graph. 
\begin{proposition}
If we evaluate $ X_{\Gamma} $ at $ x_{i} = 1$ for $i \in [\pm n] $ and $ x_{i}=0 $ for $ i \not\in [\pm n] $ (resp. $ x_{i} = 1 $ for $ i \in [\pm n]\setminus\{0\} $ and $ x_{i}=0 $ for $ i \not\in [\pm n]\setminus\{0\} $), then the value is $ \chi_{\Gamma}(2n+1) $ (resp. $ \chi^{\ast}_{\Gamma}(2n) $). 
\end{proposition}

\subsection{Power sum expansion}
Let $ \Gamma $ be a signed graph and $ S $ a subset of $ E_{\Gamma} $. 
Define $ \Gamma_{S} $ is the signed graph on vertex set $ V_{\Gamma} $ with edge set $ S $.
Let $ U $ denote the set of the vertices of the unbalanced connected components of $ \Gamma_{S} $.
Suppose that $ \Gamma_{S} $ has $ r $ balanced connected components.
The vertex set of every balanced components is decomposed as $ A_{j} \sqcup B_{j} \ (1 \leq j \leq r) $ by the following theorem. 

\begin{theorem}[Harary {\cite[Theorem 3]{harary1953notion-tmmj}}]\label{Harary}
A signed graph $ \Gamma $ is balanced if and only if there exist a decomposition $ V_{\Gamma} = A \sqcup B $ such that each positive edge joins two vertices in $ A $ or $ B $ and each negative edge is between a vertex in $ A $ and a vertex in $ B $.
\end{theorem}

Note that the decomposition $ V_{\Gamma} = U \sqcup A_{1} \sqcup B_{1} \sqcup \dots \sqcup A_{r} \sqcup B_{r} $ is determined by $ S \subseteq E_{\Gamma} $. 
Let $ u = |U| $ and for every $ j $ let $ a_{j} = |A_{j}| $ and $ b_{j} = |B_{j}| $. 
The \textbf{type} of $ S $, denoted by $ \type(S) $, is the partition $ (u,\lambda) $, where $ \lambda = \begin{pmatrix}
a_{1} & \dots & a_{r} \\
b_{1} & \dots & b_{r} 
\end{pmatrix} $. 


Define $ K_{S} $ by 
\begin{align*}
K_{S} \coloneqq \Set{\kappa \colon V_{\Gamma} \to \mathbb{Z} | \begin{array}{ll}
\kappa(v) = 0 & \text{ if } v \in U, \\
\kappa(v) = \kappa(w) & \text{ if } v,w \in A_{j} \text{ or } v,w \in B_{j} \text{ for some } j,\\
\kappa(v) = -\kappa(w) & \text{ if } v \in A_{j} \text{ and } w \in B_{j} \text{ for some } j. 
\end{array} }. 
\end{align*}

\begin{lemma}\label{power sum coloring}
Let $ S \subseteq E_{\Gamma} $. 
Then 
\begin{align*}
p_{\type(S)} = \sum_{\kappa \in K_{S}}x_{\kappa}. 
\end{align*}
\end{lemma}
\begin{proof}
Let $ u = |U| $ and $ a_{j} = |A_{j}|, b_{j} = |B_{j}| $ for each $ j \in \{1, \dots, r\} $. 
Then $ \type(S) = (u, \lambda) $, where $ \lambda = \begin{pmatrix}
a_{1} & \dots & a_{r} \\
b_{1} & \dots & b_{r} 
\end{pmatrix} $ and 
\begin{align*}
p_{\type(S)} 
&= x_{0}^{u}p_{\begin{psmallmatrix}
a_{1} \\ b_{1} \\
\end{psmallmatrix}} \cdots p_{\begin{psmallmatrix}
a_{r} \\ b_{r}
\end{psmallmatrix}} \\
&= \sum_{(i_{1}, \dots, i_{r}) \in \mathbb{Z}^{r}}x_{0}^{u}x_{i_{1}}^{a_{1}}x_{-i_{1}}^{b_{1}} \cdots x_{i_{r}}^{a_{r}}x_{-i_{r}}^{b_{r}}. 
\end{align*}
Define a map $ \phi \colon \mathbb{Z}^{r} \to K_{S} $ by
\begin{align*}
\phi(i_{1}, \dots, i_{r})(v) \coloneqq \begin{cases}
0 & \text{ if } v \in U, \\
i_{j} & \text{ if } v \in A_{j}, \\
-i_{j} & \text{ if } v \in B_{j}. 
\end{cases}
\end{align*}
One can show that $ \phi $ is bijective and hence $ p_{\type(S)} = \sum_{\kappa \in K_{S}}x_{\kappa} $. 
\end{proof}

Let $ \Gamma $ be a signed graph on $ [\ell] $ and $ \mathcal{A}_{\Gamma} $ the corresponding signed-graphic arrangement.
The \textbf{type} of $ Y \in L(\mathcal{A}_{\Gamma}) $, denoted by $ \type(Y) $, is the type of $ S $, where $ S \in L(\Gamma) $ is the corresponding flat. 

\begin{lemma}\label{power sum flat} 
For every $ Y \in L(\mathcal{A}_{\Gamma}) $, 
\begin{align*}
p_{\type(Y)} = \sum_{\boldsymbol{\alpha} \in Y \cap \mathbb{Z}^{\ell}} x_{\boldsymbol{\alpha}}. 
\end{align*}
\end{lemma}
\begin{proof}
Suppose that $ S \in L(\Gamma) $ is the flat corresponding to $ Y $. 
Then the map $ \phi \colon K_{S} \to Y \cap \mathbb{Z}^{\ell} $ determined by $ \phi(\kappa) = (\kappa(1), \dots, \kappa(\ell)) $ is a bijection. 
Therefore Lemma \ref{power sum coloring} leads to the conclusion. 
\end{proof}

Given a coloring $ \kappa \colon V_{\Gamma} \to \mathbb{Z} $, we define $ E_{\kappa} = E_{\kappa}^{+} \sqcup E_{\kappa}^{-} \sqcup L_{\kappa} $ by 
\begin{align*}
E_{\kappa}^{+} &\coloneqq \Set{\{v,w\} \in E_{\Gamma}^{+} | \kappa(v) = \kappa(w) \quad (1 \leq v < w \leq \ell) }, \\
E_{\kappa}^{-} &\coloneqq \Set{\{v,w\} \in E_{\Gamma}^{-} | \kappa(v) = -\kappa(w) \quad (1 \leq v < w \leq \ell) }, \\
L_{\kappa} &\coloneqq \Set{v \in L_{\Gamma} | \kappa(v) = 0}. 
\end{align*}

\begin{lemma}\label{coloring edgeset}
For $ S \subseteq E_{\Gamma} $ and $ \kappa \colon V_{\Gamma} \to \mathbb{Z} $, $ \kappa \in K_{S} $ if and only if $ S \subseteq E_{\kappa} $. 
\end{lemma}
\begin{proof}
First we assume $ \kappa \in K_{S} $ and show $ S \subseteq E_{\kappa} $. 
Suppose $ \{v,w\} $ is a positive edge in $ S $. 
Then both of $ v $ and $ w $ belong to $ U $ or there exists $ j $ such that both of $ v $ and $ w $ belong to $ A_{j} $ or $ B_{j} $. 
In any cases we have $ \kappa(v) = \kappa(w) $ and $ \{v,w\} \in E_{\kappa}^{+} $. 
By almost the same reason every negative edge in $ S $ belongs to $ E_{\kappa}^{-} $. 
Assume $ v $ is a loop in $ S $. 
Then $ v \in U $ and $ \kappa(v) = 0 $. 
Therefore $ v \in L_{\kappa} $. 
Thus $ S \subseteq E_{\kappa} $. 

Now we assume that $ S \subseteq E_{\kappa} $ and show $ \kappa \in K_{S} $. 
Let $ v \in U $. 
Then there exists a unbalanced cycle in the unbalanced component containing $ v $. 
Since an unbalanced cycle is a loop or a cycle containing an odd number of edges, the colors of the vertices in it must be $ 0 $. 
Since $ v $ and the cycle belong to the same component, we have $ \kappa(v) = 0 $. 
Next suppose that $ v,w \in A_{j} $ or $ v,w \in B_{j} $ for some $ j $. 
Then there exists a path on vertices in $ A_{j} \sqcup B_{j} $ joining $ v $ and $ w $. 
The path must have an even number of positive edges and hence $ \kappa(v) = \kappa(w) $. 
When $ v \in A_{j} $ and $ w \in B_{j} $ for some $ j $, by almost the same reason, $ \kappa(v) = -\kappa(w) $. 
Therefore $ \kappa \in K_{S} $. 
\end{proof}

The following theorem is an analogue of \cite[Theorem 2.5]{stanley1995symmetric-aim} and the proof is very similar. 
\begin{theorem}\label{power sum expansion subset}
Given a signed graph $ \Gamma $, 
\begin{align*}
X_{\Gamma} = \sum_{S \subseteq E_{\Gamma}}(-1)^{|S|}p_{\type(S)}. 
\end{align*}
\end{theorem}
\begin{proof}
By Lemma \ref{power sum coloring} and \ref{coloring edgeset}, 
\begin{align*}
\sum_{S\subseteq E_{\Gamma}}(-1)^{|S|}p_{\type(S)} 
= \sum_{S\subseteq E_{\Gamma}}\sum_{\kappa \in K_{S}}(-1)^{|S|}x_{\kappa}
= \sum_{\kappa \colon V_{\Gamma} \to \mathbb{Z}}\sum_{S \subseteq E_{\kappa}}(-1)^{|S|}x_{\kappa}. 
\end{align*}
Since 
\begin{align*}
\sum_{S \subseteq E_{\kappa}}(-1)^{|S|} = \begin{cases}
1 & \text{ if } E_{\kappa} = \varnothing \\
0 & \text{ if } E_{\kappa} \neq \varnothing
\end{cases}
\end{align*}
and $ E_{\kappa} = \varnothing $ if and only if $ \kappa $ is proper, 
\begin{align*}
\sum_{\kappa \colon V_{\Gamma} \to \mathbb{Z}}\sum_{S \subseteq E_{\kappa}}(-1)^{|S|}x_{\kappa}
= \sum_{\substack{\kappa \colon V_{\Gamma} \to \mathbb{Z} \\ \kappa \text{ is proper}}} x_{\kappa}
= X_{\Gamma}. 
\end{align*}
\end{proof}

There is another expansion with respect to power sum basis, which is an analogue of Theorem 2.6 and Corollary 2.7 in \cite{stanley1995symmetric-aim}. 
\begin{theorem}\label{power sum expansion flat}
Given a signed graph $ \Gamma $ on $ [\ell] $, 
\begin{align*}
X_{\Gamma} = \sum_{Y \in L(\mathcal{A}_{\Gamma})}\mu(\hat{0},Y)p_{\type(Y)} \quad \text{ and } \quad 
\overline{X}_{\Gamma} = \sum_{Y \in L(\mathcal{A}_{\Gamma})} \left|\mu(\hat{0},Y)\right| p_{\type(Y)}. 
\end{align*}
\end{theorem}
\begin{proof}
It follows immediately from Theorem \ref{XA}, \ref{oXA}, and Lemma \ref{power sum flat}. 
\end{proof}

\subsection{Relation with chromatic symmetric functions}

Recall the injection $ \iota \colon \Sym_{\mathbb{Q}} \to \SSym_{\mathbb{Q}} $ and the surjection $ \pi \colon \SSym_{\mathbb{Q}} \to \Sym_{\mathbb{Q}} $ (See Proposition \ref{injection surjection} and the above). 

Every simple graph $ (V,E) $ can be regarded as a signed graph $ (V,E,\varnothing, \varnothing) $. 
We have two functions $ X_{(V,E)} \in \Sym_{\mathbb{Q}} $ and $ X_{(V,E,\varnothing,\varnothing)} \in \SSym_{\mathbb{Q}} $. 
First we show that these functions are essentially the same functions. 

\begin{proposition}
$ \iota(X_{(V,E)}) = X_{(V,E,\varnothing,\varnothing)} $. 
\end{proposition}
\begin{proof}
This is immediately from \cite[Theorem 2.5]{stanley1995symmetric-aim} and Theorem \ref{power sum expansion subset}. 
\end{proof}

Given a signed graph $ \Gamma $, we define the \textbf{positive simple graph} of $ \Gamma $ by $ \Gamma^{+} \coloneqq (V_{\Gamma}, E_{\Gamma}^{+}) $. 

\begin{proposition}\label{projection}
$ \pi(X_{\Gamma}) = X_{\Gamma^{+}} $. 
\end{proposition}
\begin{proof}
Let $ S \subseteq E_{\Gamma} $. 
When $ S $ contains a negative edge or a loop, $ \pi(p_{\type(S)})=0 $. 
Otherwise, $ S \subseteq E_{\Gamma}^{+} $. 
Therefore the assertion holds by \cite[Theorem 2.5]{stanley1995symmetric-aim} and Theorem \ref{power sum expansion subset}. 
\end{proof}

\subsection{Combinatorial reciprocity (Proof of Theorem \ref{main theorem reciprocity})}

\begin{lemma}[Zaslavsky {\cite[Theorem 5.1(j)]{zaslavsky1982signed-dam}}]\label{Zaslavsky rank}
The rank of a flat $ S \in L(\Gamma) $ equals $ |V_{\Gamma}| - b(S) $, where $ b(S) $ denotes the number of balanced connected components of $ \Gamma_{S} $. 
In particular, if $ \Gamma $ is a signed graph on $ [\ell] $, then for every $ Y \in L(\mathcal{A}_{\Gamma}) $, we have $ \codim Y = \ell - b(Y) $, where $ b(Y) = b(S) $ and $ S \in L(\Gamma) $ denotes the corresponding flat. 
\end{lemma}

\begin{theorem}[Restatement of Theorem \ref{main theorem reciprocity}]
Let $ \Gamma $ be a signed graph. 
Then $ \omega X_{\Gamma} = \overline{X}_{\Gamma} $. 
\end{theorem}
\begin{proof}
Using Lemma \ref{Rota sign theorem}, Theorem \ref{power sum expansion flat}, and Lemma \ref{Zaslavsky rank}, we obtain 
\begin{align*}
\omega X_{\Gamma} &= \sum_{Y \in L(\mathcal{A}_{\Gamma})}\mu(\hat{0},Y)\omega p_{\type(Y)} 
= \sum_{Y \in L(\mathcal{A}_{\Gamma})}\mu(\hat{0},Y)(-1)^{\ell-b(Y)} p_{\type(Y)} \\
&= \sum_{Y \in L(\mathcal{A}_{\Gamma})}\mu(\hat{0},Y)(-1)^{\codim Y} p_{\type(Y)} 
= \sum_{Y \in L(\mathcal{A}_{\Gamma})}\left|\mu(\hat{0},Y)\right| p_{\type(Y)} \\
&= \overline{X}_{\Gamma}. 
\end{align*}
\end{proof}

\begin{remark}
We can recover Stanley's combinatorial reciprocity (Theorem \ref{Stanley CR chromatic symmetric function}) by Proposition \ref{injection surjection} and \ref{projection}, and Theorem \ref{main theorem reciprocity}
\end{remark}

\subsection{Connectedness and irreducibility}
Cho and Willigenburg \cite{cho2016chromatic-tejoc} constructed systems of generators for $ \Sym_{\mathbb{Q}} $ consisting of chromatic symmetric functions of connected simple graphs. 
They essentially proved that a simple graph is connected if and only if the chromatic symmetric function is irreducible (See also \cite[Corollary 2.4]{tsujie2018chromatic-gac}). 
We can prove an analogue for chromatic signed-symmetric functions as follows. 

\begin{proposition}
Let $ \Gamma $ be a signed graph. 
Then $ \Gamma $ is connected if and only if the chromatic signed-symmetric function $ X_{\Gamma} $ is irreducible. 
\end{proposition}
\begin{proof}
If $ X_{\Gamma} $ is irreducible, then $ \Gamma $ is connected by Proposition \ref{disjoint union product}. 
Now suppose that $ \Gamma $ is connected. 
Let $ S $ be a maximal balanced flat of $ \Gamma $. 
Since $ \Gamma $ is connected, $ \Gamma_{S} $ is connected. 
Therefore $ \type(S) $ is of the form $ \begin{pmatrix}
a \\ b
\end{pmatrix} $. 
By Lemma \ref{Rota sign theorem} and Theorem \ref{power sum expansion flat}, $ X_{\Gamma} $ contains a term $ \mu(\hat{0}, S)p_{\begin{psmallmatrix}
a \\ b
\end{psmallmatrix}} $ with $ \mu(\hat{0}, S) \neq 0 $. 
If $ T $ is another maximal balanced flat, then $ \mu(\hat{0}, S) $ and $ \mu(\hat{0}, T) $ have the same sign by Lemma \ref{Rota sign theorem} and \ref{Zaslavsky rank}. 
Therefore $ X_{\Gamma} $ contains a power sum signed-symmetric function of the form $ p_{\begin{psmallmatrix}
a \\ b
\end{psmallmatrix}} $, which cannot be obtained from a product of two signed-symmetric functions by Theorem \ref{power sum basis}. 
Thus $ X_{\Gamma} $ is irreducible. 
\end{proof}

\subsection{Signed paths (Proof of Theorem \ref{main signed paths})}

\begin{proposition}\label{signed tree components}
Let $ T $ be a signed tree. 
The number of subsets $ S \subseteq E_{T} $ such that $ T_{S} $ has $ r $ connected components and $ p $ connected components containing no negative edges is obtained from $ X_{T} $. 
Especially the type of $ E_{T} $ is determined by $ X_{T} $. 
\end{proposition}
\begin{proof}
Let $ S \subseteq E_{T} $. 
Then every connected component of $ T_{S} $ is balanced. 
Suppose that $ T_{S} $ has $ r $ connected components and $ p $ connected components containing no negative edges. 
Then $ \type(S) $ is of the form $ \begin{pmatrix}
a_{1} & \cdots & a_{p} & a_{p+1} & \cdots & a_{r} \\
0 & \cdots & 0 & b_{p+1} & \cdots & b_{r}
\end{pmatrix} $ with $ a_{i} \neq 0 \ (1 \leq i \leq r) $ and $ b_{i} \neq 0 \ (p+1 \leq i \leq r) $ and $ |S| = |V_{T}|-r $. 

Let $ X_{T} = \sum_{\lambda}c_{\lambda}p_{\lambda} $. 
Theorem \ref{power sum expansion subset} shows that the sum $ \sum_{\lambda}|c_{\lambda}| $, where $ \lambda $ runs over all partitions of the form $ \begin{pmatrix}
a_{1} & \cdots & a_{p} & a_{p+1} & \cdots & a_{r} \\
0 & \cdots & 0 & b_{p+1} & \cdots & b_{r}
\end{pmatrix} $ with $ a_{i} \neq 0 \ (1 \leq i \leq r) $ and $ b_{i} \neq 0 \ (p+1 \leq i \leq r) $, is the desired number. 

Moreover $ T_{S} $ is connected if and only if the type of $ S $ is of the form $ \begin{pmatrix}
a \\ b
\end{pmatrix} $. 
This happens only when $ S = E_{T} $ and hence the type of $ E_{T} $ is obtained from $ X_{T} $. 
\end{proof}

Let $ \boldsymbol{\alpha} = (\alpha_{1}, \dots, \alpha_{\ell}) \in \mathbb{Z}_{>0}^{\ell} $ denote an integer composition. 
Recall $ P_{\boldsymbol{\alpha}} $ denotes the signed path obtained by connecting paths consisting of positive edges $ P_{\alpha_{1}}, \dots, P_{\alpha_{\ell}} $ with negative edges in this order. 
Note that $ P_{\boldsymbol{\alpha}} $ is isomorphic to $ P_{\boldsymbol{\beta}} $ if and only if $ \boldsymbol{\alpha} \in \{\boldsymbol{\beta}, \boldsymbol{\beta}^{r}\} $, where $ \boldsymbol{\beta}^{r} \coloneqq (\beta_{m}, \dots, \beta_{1}) $ if the \textbf{reverse} of a composition $ \boldsymbol{\beta} = (\beta_{1}, \dots, \beta_{m}) $. 

\begin{lemma}\label{signed path partition type}
Let $ \boldsymbol{\alpha} = (\alpha_{1}, \dots, \alpha_{\ell}) $ and $ \boldsymbol{\beta} = (\beta_{1}, \dots, \beta_{m}) $ be compositions. 
Suppose that $ X_{P_{\boldsymbol{\alpha}}} = X_{P_{\boldsymbol{\beta}}} $. 
Then $ \ell = m $ and there exists a permutation $ \sigma $ such that $ \alpha_{i} = \beta_{\sigma(i)} $ for any $ i \in \{1, \dots \ell\} $. 
\end{lemma}
\begin{proof}
Since $ X_{P_{\boldsymbol{\alpha}}} = X_{P_{\boldsymbol{\beta}}} $, we have $ X_{P_{\boldsymbol{\alpha}}^{+}} = X_{P_{\boldsymbol{\beta}}^{+}} $ by Proposition \ref{projection}. 
The positive simple graph $ P_{\boldsymbol{\alpha}}^{+} $ is the disjoint union $ P_{\boldsymbol{\alpha}}^{+} = P_{\alpha_{1}} \sqcup \dots \sqcup P_{\alpha_{\ell}} $. 
Hence we have $ X_{P_{\boldsymbol{\alpha}}^{+}} = X_{P_{\alpha_{1}}} \cdots X_{P_{\alpha_{\ell}}} $ and every $ X_{\alpha_{i}} $ is irreducible. 
The same holds for $ P_{\boldsymbol{\beta}} $. 

Since $ \Sym_{\mathbb{Q}} $ is a unique factorization domain, we have $ \ell = m $ and there exists a permutation $ \sigma $ such that $ X_{P_{\alpha_{i}}} = X_{P_{\beta_{\sigma(i)}}} $ for any $ i \in \{1, \dots, \ell\} $. 
The degrees of the both sides yield $ \alpha_{i} = \beta_{\sigma(i)} $. 
\end{proof}

\begin{lemma}\label{signed path length at most 4}
Suppose that $ X_{P_{\boldsymbol{\alpha}}} = X_{P_{\boldsymbol{\beta}}} $ and the length of $ \alpha $ is less than or equal to $ 4 $. 
Then $ P_{\boldsymbol{\alpha}} $ and $ P_{\boldsymbol{\beta}} $ are isomorphic. 
\end{lemma}
\begin{proof}
If the length of $ \alpha $ is $ 1 $, then the assertion holds trivially. 
Suppose that the length of $ \alpha $ is $ 2 $. 
Since $ P_{(\alpha_{1}, \alpha_{2})} $ is isomorphic to $ P_{(\alpha_{2}, \alpha_{1})} $, the assertion holds by Lemma \ref{signed path partition type}. 

Next we suppose that the length of $ \alpha $ is $ 3 $. 
Let $ S \subseteq E_{T} $ such that $ T_{S} $ consists of $ 2 $ connected components and the both components contain a negative edge. 
Such a subset $ S $ is obtained only by $ S = E_{T}\setminus\{e\} $, where $ e $ is an edge of $ P_{\alpha_{2}} $ and hence the number of such subsets is $ \alpha_{2}-1 $. 
Therefore by Proposition \ref{signed tree components} and Lemma \ref{signed path partition type}, we have $ \alpha_{2}-1=\beta_{2}-1 $ and $ \boldsymbol{\alpha} \in \{\boldsymbol{\beta}, \boldsymbol{\beta}^{r}\} $. 
Thus $ P_{\boldsymbol{\alpha}} $ is isomorphic to $ P_{\boldsymbol{\beta}} $. 

Finally we assume that the length of $ \alpha $ is $ 4 $. 
Let $ S \subseteq E_{T} $. 
The subgraph $ T_{S} $ consists of $ 3 $ connected components which contain a negative edge if and only if there exists $ e \in E_{P_{\alpha_{2}}} $ and $ e^{\prime} \in E_{P_{\alpha_{3}}} $ such that $ S = E_{T}\setminus\{e,e^{\prime}\} $. 
Therefore the number of such a subset $ S $ is $ (\alpha_{2}-1)(\alpha_{3}-1) $. 

The subgraph $ T_{S} $ consists of $ 2 $ connected components which contain a negative edge if and only if $ S=E_{T}\setminus\{e\} $, where $ e $ is an edge of $ P_{\alpha_{2}} $ or $ P_{\alpha_{3}} $ or $ e $ is the edge between $ P_{\alpha_{2}} $ and $ P_{\alpha_{3}} $. 
Hence the number of such a subset $ S $ is $ (\alpha_{2}-1) + (\alpha_{3}-1) + 1 = \alpha_{2}+\alpha_{3}-1 $. 

By Proposition \ref{signed tree components} and Lemma \ref{signed path partition type}, we have $ (\alpha_{2}-1)(\alpha_{3}-1) = (\beta_{2}-1)(\beta_{3}-1) $ and $ \alpha_{2}+\alpha_{3}-1 = \beta_{2}+\beta_{3}-1 $. 
Hence $ \alpha_{2}+\alpha_{3} = \beta_{2}+\beta_{3} $ and $ \alpha_{2}\alpha_{3} = \beta_{2}\beta_{3} $. 
Therefore $ \{\alpha_{2}, \alpha_{3}\} = \{\beta_{2}, \beta_{3}\} $. 
Without loss of generality, we may assume that $ \alpha_{2}=\beta_{2} $ and $ \alpha_{3}=\beta_{3} $. 

When $ \alpha_{1} = \beta_{1} $ and $ \alpha_{4} = \beta_{4} $, we have nothing to prove. 
Hence we assume that $ \alpha_{1} = \beta_{4} $ and $ \alpha_{4} = \beta_{1} $. 
By Proposition \ref{signed tree components}, the type $ \begin{pmatrix}
\alpha_{1} + \alpha_{3} \\
\alpha_{2} + \alpha_{4}
\end{pmatrix} $ coincides with $ \begin{pmatrix}
\beta_{1} + \beta_{3} \\
\beta_{2} + \beta_{4}
\end{pmatrix} = \begin{pmatrix}
\alpha_{4} + \alpha_{3} \\
\alpha_{2} + \alpha_{1}
\end{pmatrix} $. 

If $ \alpha_{1}+\alpha_{3} = \alpha_{4}+\alpha_{3} $, then $ \alpha_{1} = \alpha_{4} $ and if $ \alpha_{1}+\alpha_{3} = \alpha_{2}+\alpha_{1} $, then $ \alpha_{2}=\alpha_{3} $. 
In the both cases, $ \boldsymbol{\alpha} \in \{\boldsymbol{\beta}, \boldsymbol{\beta}^{r}\} $. 
Therefore $ P_{\boldsymbol{\alpha}} $ is isomorphic to $ P_{\boldsymbol{\beta}} $. 
\end{proof}

\begin{lemma}\label{signed path unimodal}
Suppose that $ X_{P_{\boldsymbol{\alpha}}} = X_{P_{\boldsymbol{\beta}}} $ and both of $ \boldsymbol{\alpha} $ and $ \boldsymbol{\beta} $ are unimodal.  
Then $ P_{\boldsymbol{\alpha}} $ is isomorphic to $ P_{\boldsymbol{\beta}} $. 
\end{lemma}
\begin{proof}
We proceed by induction on the length $ \ell $ of $ \alpha $. 
Clearly, the assertion holds for $ \ell = 1 $. 
Hence we assume $ \ell \geq 2 $.  
Since $ \boldsymbol{\alpha} = (\alpha_{1}, \dots, \alpha_{\ell}) $ is unimodal, there exist an index $ t $ and a positive integer $ k $ such that $ \alpha_{1} \leq \dots \leq \alpha_{t-1} < \alpha_{t} = \dots = \alpha_{t+k} > \alpha_{t+k+1} \geq \dots \geq \alpha_{\ell} $. 
Let $ \boldsymbol{\alpha}^{\prime} = (\alpha_{1}, \dots, \alpha_{t-1}) $ and $ \boldsymbol{\alpha}^{\prime\prime} = (\alpha_{t+k+1}, \dots, \alpha_{\ell}) $. 
Then by Theorem \ref{power sum expansion subset} there exists $ f \in \SSym_{\mathbb{Q}} $ such that 
\begin{align*}
X_{P_{\boldsymbol{\alpha}}} = p_{\begin{psmallmatrix}
\alpha_{t} \\ 0
\end{psmallmatrix}}^{k} \ X_{P_{\boldsymbol{\alpha}^{\prime}}\sqcup P_{\boldsymbol{\alpha}^{\prime\prime}}} + f \qquad \text{ and } \qquad p_{\begin{psmallmatrix}
\alpha_{t} \\ 0
\end{psmallmatrix}}^{k} \nmid f. 
\end{align*}
Since $ X_{P_{\boldsymbol{\alpha}}} = X_{P_{\boldsymbol{\beta}}} $ and $ \boldsymbol{\beta} $ is also unimodal, we have $ X_{P_{\boldsymbol{\alpha}^{\prime}}\sqcup P_{\boldsymbol{\alpha}^{\prime\prime}}} = X_{P_{\boldsymbol{\beta}^{\prime}}\sqcup P_{\boldsymbol{\beta}^{\prime\prime}}} $, where $ \boldsymbol{\beta}^{\prime} = (\beta_{1}, \dots, \beta_{s-1}) $ and $ \boldsymbol{\beta}^{\prime\prime} = (\beta_{s+k+1}, \dots, \beta_{\ell}) $ for some $ s $ and $ \beta_{s} = \dots = \beta_{s+k} = \alpha_{t} $. 
Therefore $ X_{P_{\boldsymbol{\alpha}^{\prime}}}X_{P_{\boldsymbol{\alpha}^{\prime\prime}}} = X_{P_{\boldsymbol{\beta}^{\prime}}}X_{P_{\boldsymbol{\beta}^{\prime\prime}}} $. 

Without loss of generality, we may assume that $ X_{P_{\boldsymbol{\alpha}^{\prime}}} = X_{P_{\boldsymbol{\beta}^{\prime}}} $ and $ X_{P_{\boldsymbol{\alpha}^{\prime\prime}}} = X_{P_{\boldsymbol{\beta}^{\prime\prime}}} $ and will prove that $ \alpha = \beta $. 
By the induction hypothesis, $ P_{\boldsymbol{\alpha}^{\prime}} $ is isomorphic to $ P_{\boldsymbol{\beta}^{\prime}} $ and $ P_{\boldsymbol{\alpha}^{\prime\prime}} $ is isomorphic to $ P_{\boldsymbol{\beta}^{\prime\prime}} $. 
Since these compositions are monotonic, $ \boldsymbol{\alpha}^{\prime} = \boldsymbol{\beta}^{\prime} $ and $ \boldsymbol{\alpha}^{\prime\prime} = \boldsymbol{\beta}^{\prime\prime} $. 
Therefore $ \boldsymbol{\alpha} = \boldsymbol{\beta} $. 
\end{proof}

\begin{theorem}[Restatement of Theorem \ref{main signed paths}]
Suppose that $ X_{P_{\boldsymbol{\alpha}}} = X_{P_{\boldsymbol{\beta}}} $. 
If one of the following conditions holds, then $ P_{\boldsymbol{\alpha}} $ and $ P_{\boldsymbol{\beta}} $ are isomorphic. 
\begin{enumerate}[(1)]
\item The length of $ \boldsymbol{\alpha} $ is less than or equal to $ 4 $. 
\item $ \boldsymbol{\alpha} $ and $ \boldsymbol{\beta} $ are unimodal. 
\end{enumerate}
\end{theorem}
\begin{proof}
This is a direct consequence of Lemma \ref{signed path length at most 4} and \ref{signed path unimodal}. 
\end{proof}

\bibliographystyle{amsplain}
\bibliography{bibfile}


\end{document}